\documentclass[12pt]{amsart}

\usepackage{amssymb}
\usepackage[colorlinks,citecolor=blue,linkcolor=magenta]{hyperref}

\usepackage{enumitem}
\usepackage{url}

\usepackage{graphicx}

\makeatletter
\@namedef{subjclassname@2010}{%
  \textup{2010} Mathematics Subject Classification}
\makeatother

\usepackage[T1]{fontenc}

\newtheorem{theorem}{Theorem}[section]
\newtheorem{corollary}[theorem]{Corollary}

\theoremstyle{definition}

\newtheorem{remark}[theorem]{Remark}

\newtheorem{conjecture}[theorem]{Conjecture}

\numberwithin{equation}{section}
 
\newcommand\m[1]{\mathbb{#1}}
\newcommand\mc[1]{\mathcal{#1}}

\newcommand\inn[1]{\langle #1 \rangle}

\begin{document}

\baselineskip=17pt

\title[Invertible elements in Fourier-Stieltjes algebras]{On the structure of invertible elements in certain Fourier-Stieltjes algebras}

\author[A. Thamizhazhagan]{Aasaimani Thamizhazhagan}
\address{Department of Pure Mathematics\\ University of Waterloo\\
Ontario\\Canada N2L 3G1}
\email{athamizh@uwaterloo.ca}

\date{}

\begin{abstract}
For a locally compact abelian group $G$, J. L. Taylor (1971) gave a complete characterization of invertible elements in the measure algebra $M(G)$. Using the Fourier-Stieltjes transform, this characterization can be carried out in the context of Fourier-Stieltjes algebras $B(G)$. We obtain this latter characterization for the Fourier-Stieltjes algebra $B(G)$ of certain classes of locally compact groups, in particular, many totally minimal groups and the $ax+b$-group. 
\end{abstract}



\maketitle

\section{Introduction}\label{intro}
For a non-discrete locally compact abelian group $\Gamma$ with dual group $G$, Taylor (\cite[Theorem 3]{Tay1}) proved a factorization theorem for invertible measures in its measure algebra $M(\Gamma)$: for each $\mu \in M(\Gamma)^{-1}$, there are l.c.a.\ groups $\Gamma_{\tau_1}, \ldots,\Gamma_{\tau_n}$ continuously isomorphic to $\Gamma$ and measures $\nu_i \in (L^1(\Gamma_{\tau_i}) \oplus \m C 1)^{-1}, \ \omega \in M(\Gamma),$ such that $$ \mu = \nu_1 * \cdots * \nu_n * \exp(\omega). $$ The measures $\nu_i$ are unique modulo $\exp(L^1(\Gamma_{\tau_i}) \oplus \m C 1)$. By calling $\Gamma_\tau$ continuously isomorphic to $\Gamma$, we mean that $\Gamma_\tau$ is equal to $\Gamma$ as a group, but it is an l.c.a.\ group under a topology $\tau$ possibly finer than that of  $\Gamma$. 

The non-commutative analogues of $M(\Gamma)$ and $L^1(\Gamma)$ for general locally compact groups $G$ are the Fourier-Stieltjes algebra $B(G)$ and the Fourier algebra $A(G)$, introduced by Eymard (\cite{Eym}). When $G$ is abelian with dual group $\Gamma$, the Fourier-Stieltjes transformation maps $M(\Gamma)$ isometrically onto $B(G)$ and $L^1(\Gamma)$ onto $A(G)$. Via the Fourier-Stieltjes transform for abelian $G$, we see that every invertible element $u \in B(G)$ is of the form $u=v_1\cdots v_n\cdot \exp(w)$ where $v_i \in (A(G_{\sigma_i}) \oplus \m C 1)^{-1}, w \in B(G)$ and the topologies $\sigma_i$ are coarser than the ambient group topology.

 For a non-discrete locally compact abelian group $\Gamma$, several factors such as the Wiener-Pitt Phenomenon, the assymmetricity of $M(\Gamma)$ and the existence of independent Cantor sets indicate that the spectrum of $M(\Gamma)$ is quite complicated, as is the problem of deciding when $\mu \in M(\Gamma)$ is invertible (i.e. when $\widehat{\mu}$ does not vanish on $\Delta(M(\Gamma))$). In \cite{Tay1}, with varying degrees of technical difficulty, Taylor reduces the problem of invertibility in $M(\Gamma)$ to a symmetric subalgebra $\mathcal{L}(\Gamma)$, the closed linear span of all maximal subalgebras which are isometrically isomorphic to some locally compact abelian group algebra. He calls this the $spine$ of $M(\Gamma)$. The key to Taylor's approach is to notice that $\Delta(M(\Gamma))$ has a great deal of structure not generally enjoyed by maximal ideal spaces. It has a semigroup structure, an order structure, and has certain subsets on which two important topologies, the weak and the strong topology, coincide. Taylor introduced the notion of critical points, those positive elements in $\Delta(M(\Gamma))$ which cannot be weakly approximated by strictly smaller positive elements of $\Delta(M(\Gamma))$. He establishes a one-to-one correspondence between the maximal subalgebras and critical points. These critical points have to be idempotents and the $G_{\sigma_i}$'s above are exactly the maximal groups in $\Delta(M(\Gamma))$ at these critical points, thereby locally compact abelian groups (\cite[Proposition 2.3. (4)]{Tay1}). 

 In the hope of extending Taylor's theory to the context of Fourier-Stieltjes algebras, two different approaches have been attempted so far: spectral analysis and topological analysis. In \cite{Wal2}, Walter showed that the maximal ideal space of $B(G)$ for a general locally compact group $G$ exhibits a structure like that of an abelian measure algebras and extended the notion of critical points to $\Delta(B(G))$. Then a reasonable conjecture would be that every invertible element $u$ in $B(G)$ is of the form $u = v_1\cdots v_n\cdot \exp(w)$ where $w \in B(G), \ v_i \in (A(G_z) \oplus \m C 1)^{-1}$ for some central critical element $z$ in $\Delta(B(G))$. Here, $G_z$ is a locally compact group about the critical point $z$. On the other hand, in \cite{Nicmon}, Ilie and Spronk develop the non-commutative dual analogue of the spine of an abelian measure algebra and (\cite{Nicmon}, Theorem 5.1) provides an explicit description of the above topologies $\sigma_i$. Moreover, they note that it would be interesting to determine if invertible elements $u$ in $B(G)$ are of the form $u = v_1\cdots v_n\cdot \exp(w)$ where $w \in B(G), \ v_i \in (A(G_{\rho_i}) \oplus \m C 1)^{-1}$ for some locally precompact topology $\rho_i$ on $G$. By $G_\rho$, we mean the locally compact completion of $G$ with respect to $\rho$.
 
We prove that the above conjectures hold for Fourier-Stieltjes algebra that satisfy some very restrictive conditions. These conditions are satisfied for interesting examples within Lie groups, in particular, many connected totally minimal groups. Moreover, elementary arguments based on the Arens-Royden theorem immediately allow for a deeper analysis of invertible elements in motion groups and the $ax+b$-group. We note that these examples cover those for which \cite{Nicmon} have constructed the spine. 

  We briefly describe the structure of this paper: in Section \ref{Not}, we establish notation and recall some of the necessary background. In Section \ref{aroy}, we observe a consequence (Corollary \ref{arcor}) of the Arens-Royden Theorem (\cite{Arens}, \cite{Roy}). In Section \ref{Rajchinv}, as an application of Corollary $\ref{arcor}$, we obtain that the abstract index group of the regular Rajchman algebra $B_0(G)$ is exactly the abstract index group of the Fourier algebra $A(G)$ (Remark \ref{reduc}). Here $B_0(G)$ is the set of all functions in $B(G)$ that vanish at infinity. We end this section with another observation, which gives the desired characterization of the invertible elements in the case of Euclidean motion groups, $p$-adic motion groups and the $ax+b$ group. In Section $\ref{main}$, we prove our main result (Theorem \ref{mainthm}). In Section \ref{decom}, we consider several classes of locally compact groups $G$ that include many totally minimal groups, connected semisimple groups with finite centre and extensions of compact analytic group by a nilpotent analytic group. We deduce from \cite{May1}, \cite{May2}, \cite{May3}, \cite{Cowl}, \cite{Liu} that their Fourier-Stieltjes algebras are symmetric, with their Gelfand spectrum being a semilattice of groups and satisfies the hypothesis of Theorem $\ref{mainthm}$. 

\section{Notation and Background}\label{Not}
Let $G$ be a locally compact group. The $Fourier$-$Stieltjes$ $algebra$ $B(G)$ is the linear span of the set $P(G)$ of all continuous positive definite functions on $G$ and can be identified with the Banach space dual of the full group $C^*$-algebra. For $u \in B(G)$ and $f \in L^1(G)$, the pairing is given by $\langle u,f\rangle = \int_G f(x)u(x)dx$. The space $B(G)$ is a involutive (semisimple) commutative Banach algebra under pointwise multiplication. Every $u \in B(G)$ is the coefficient function of some unitary representation of $G$, meaning there exist a unitary representation $\pi$ of $G$ and $\xi,\eta \in \mathcal{H}_\pi, $ the Hilbert space of $\pi$, such that $u(x) = \langle \pi(x)\xi,\eta\rangle$ for all $x \in G$ and $\|u\| = \|\xi\|.\|\eta\|$.

The $Fourier$ $algebra$ $A(G)$ is the closed ideal of $B(G)$ generated by all compactly supported functions in $B(G)$. Every $u \in A(G)$ is the coefficient function of the left regular representation $\lambda$ of $G$. The spectrum or Gelfand space $\Delta(A(G))$ of $A(G)$ can naturally be identified with $G$. More precisely, the map $x \to \phi_x$, where $\phi_x(u) = u(x)$ for $u\in A(G)$, is a homeomorphism between $G$ and $\Delta(A(G))$. 

When $G$ is abelian and $\widehat{G}$ the dual group of $G$, then the inverse Fourier and Fourier-Stieltjes transforms map $A(G)$ and $B(G)$ isometrically isomorphic onto the dual group algebra $L^1(\widehat{G})$ and the measure algebra $M(\widehat{G})$, respectively. 

Denote by $A_\pi$, the closure in $B(G)$ of the linear span of all matrix cofficients of the unitary representation $(\pi, \mathcal{H}_\pi)$. By Arsac's theory (\cite[3.1.II]{Ars}), $A_\pi$ characterizes $(\pi, \mathcal{H}_\pi)$ up to quasi-equivalence.

The $Rajchman$ $algebra$ $B_0(G)$ is defined as the intersection $$B_0(G) = B(G) \cap C_0(G).$$ We have the inclusion: $A(G) \subseteq B_0(G)$ and in the sequel, we always identify $\Delta(A(G))$ with $G$ and view $G$ to be embedded into $\Delta(B_0(G))$ and $\Delta(B(G))$ as evaluations. 

We let $W^*(G) = B(G)^*$, the von Neumann algebra generated by the $universal$ $unitary$ $representation$ ${\varpi}_G$ on $G$. By Arsac's theory (\cite[3.17]{Ars}), any closed translation-invariant subspace $A$ of $B(G)$ is of the form $$A = Z\cdot B(G)$$ where $Z \in W^*(G)$ is a central projection and $$ Z\cdot \inn{\pi(.)\xi,\eta} = \inn{\pi(.)Z\xi,\eta}.$$

 We end this section with certain facts about $\Delta$, the spectrum of $B(G)$. For $s \in \Delta$ there are naturally associated two endomorphisms of $B(G)$, $\gamma_s: b \in B(G) \to s.b \in B(G)$ and $\delta_s: b \in B(G) \to b.s \in B(G)$, where $\inn{s.b, x} = \inn{b,xs}$ and $\inn{b.s,x} = \inn{b,sx}$ for all $x \in W^*(G).$ Let $\Delta_+$ denote the positive, Hermitian elements and $\Delta_p$ the self-adjoint idempotents in $\Delta$. If $G$ is not compact, then $\Delta$ contains partial isometries and projections different from $e$, the identity of $G$. In what follows, let $\Sigma(G)$ denote the class of all continuous unitary representations of $G$ modulo unitary equivalence. We remark that any continuous unitary representation $(\pi,\mc H _\pi)$ of $G$ extends uniquely to a normal (i.e. $\sigma(W^*(G),B(G))$ continuous) representation of $W^*(G)$.  
\begin{theorem}\label{Walty}
Let $G$ be a locally compact group. Then
\begin{enumerate}
    \item \emph{(\cite[Theorem 2]{Wal2})}$z_F = \sup\{z[\pi]: z[\pi] \text{ is the support in } W^*(G)$ of finite dimensional (unitary) representation $\pi \}$. Then $z_F$ is a central projection in $W^*(G)$, and $z_F \in \Delta_+$. Moreover if $s \in \Delta_+$, we have $z_Fs = z_F$, i.e., $z_F \leq s.$ \label{ap}
    \item \emph{(\cite[Proposition 8]{Wal2})}\label{AP} $G_{z_F} = \{s \in \Delta : ss^* = s^*s = z_F\}$ is the almost periodic compactification of $G$, and $B(G_{z_F})$ is isometrically isomorphic to $B(G)\cap AP(G) = z_F.B(G)$, where $AP(G)$ denotes the $C^*$-algebra of almost periodic functions on $G$. 
\end{enumerate}
\end{theorem}

\section{A Consequence of the Arens-Royden Theorem}\label{aroy}
If $A$ is a commutative Banach algebra with identity, we denote by $A^{-1}$ the multiplicative group of all invertible elements of $A$. If $a = e^b$ for some $b \in A$, then $a \in A^{-1}$ with $a^{-1} = e^{-b}$. Hence the group $\exp{A}$ of all elements of $A$ with logarithms is a subgroup of $A^{-1}$. If $A^{-1}$ is given the norm topology, then it is a topological group with $\exp{A}$ as an open subgroup. Note that if $a = e^b \in \exp{A}$, then $t \to e^{tb}$ $(t \in [0,1])$ yields an arc in $\exp(A)$ connecting $a$ to $1$. In fact, $\exp(A)$ is exactly the connected component of the identity in $A^{-1}$, and analytic functional calculus shows that any element within distance $1$ of the identity is in this component. We denote the discrete group $A^{-1}/\exp(A)$ by $H^1(A)$.

Let $\Delta(A)$ be the Gelfand spectrum of $A$. We denote the group $$H^1(C(\Delta(A))) = \displaystyle \frac{C(\Delta(A))^{-1}}{\exp(C(\Delta(A)))}$$ by $H^1(\Delta(A))$. The Gelfand transform $a \to \widehat{a}: A \to C(\Delta(A))$ where $\widehat{a}(\gamma) = \inn{a,\gamma}$ for $a \in A$ and $\gamma \in \Delta(A)$, maps $A^{-1}$ into $C(\Delta(A))^{-1}$ and $\exp(A)$ into $\exp(C(\Delta(A)))$. In fact, for any unital Banach algebra homomorphism $\phi: A \to B$, we have $\phi A^{-1} \subset B^{-1}$ and $\phi \exp(A) \subset \exp(B)$; thus $\phi$ induces a map $\phi^*$ of $H^1(A)$ into $H^1(B)$. In particular, $a \to \widehat{a}$ induces a homomorphism of $H^1(A)$ into $H^1(\Delta(A))$.

\begin{theorem}[Arens and Royden, \cite{Arens}, \cite{Roy}]\label{ar}
If $A$ is a commutative Banach algebra with identity $1$ and spectrum $\Delta(A)$, then the map $H^1(A) \to H^1(\Delta(A)) = H^1(C(\Delta(A)))$ induced by the Gelfand transform, is an isomorphism. 
\end{theorem}

The following result may be well known, though it is not explicitly found in the literature, so we provide a proof. 

\begin{corollary}\label{arcor}
Let $A \subset B$ be a Banach sub-algebra such that $\Delta(A) = \Delta(B)$ and $1 \in A$. Then $B^{-1} = \exp(B) A^{-1}$. 
\end{corollary}
\begin{proof}
Let $\Delta = \Delta(A) = \Delta(B)$. By Theorem $\ref{ar}$, if $\Gamma_j$ $(j = A,B)$ denotes the Gelfand transform of $j$, we see that $(\Gamma_B^*)^{-1} \circ \Gamma_A^*: H^1(A) \to H^1(B) $ is an isomorphism. Since the inclusion $i: A \to B$ is a unital Banach algebra homomorphism, it induces a map $i^*: H^1(A) \to H^1(B)$ such that $ \Gamma_B^* \circ i^* = \Gamma_A^*$. Now $i^*$ is injective because $\Gamma_B^* \circ i^*$ is injective. It is enough to show that $i^*$ is surjective. Suppose not. Then there exists $b\in B$ such that $b \exp(B) \notin i^*\left(H^1(A)\right)$ and $\Gamma_B^*(b \exp B) = \widehat{b}\exp(C(\Delta))$. Let $d \in A$ be such that $(\Gamma_A^*)^{-1}\left(\widehat{b}\exp(C(\Delta)\right) = d \exp A$. Since $\Gamma_B^*(d \exp B) =  \Gamma_B^* \circ i^*(d \exp A) = \Gamma_A^*(d \exp A)$, we have $\widehat{d} \exp(C(\Delta)) = \widehat{b} \exp (C(\Delta))$. But then the fact that $\Gamma_B^*$ is an isomorphism implies that $b\exp B = d \exp B \in i^*\left(H^1(A)\right)$, a contradiction.
\end{proof}

\section{Regularity of Rajchmann algebras and the invertibles }\label{Rajchinv} Recall that $A$ is called $regular$ if given any closed subset $E$ of $\Delta(A)$ and $\gamma \in \Delta(A) \backslash E$, there exists $a \in A$ such that $\widehat{a} = 0$ on $E$ and $\widehat{a}(\gamma) \neq 0$. Recall also that $A(G)$ is regular for any locally compact group $G$. The next theorem characterizes the regularity of $B_0(G)$ for general locally compact groups $G$. Note that in the case of abelian $G$, the algebra $B_0(G)$ is regular only if $G$ is compact. 
 
 \begin{theorem}(\cite[Theorem 2.1]{kanul})
 Let $A$ be any closed subalgebra of $B_0(G)$ containing $A(G)$. Then 
 \begin{enumerate}
     \item $G$ (i.e. $\Delta(A(G))$) is closed in $\Delta(A)$.
     \item $A$ is regular if and only if $\Delta(A) = G$. 
 \end{enumerate}
 In particular, $B_0(G)$ is regular if and only if $\Delta(B_0(G)) = G$
 \end{theorem}

\begin{remark}\label{reduc}

 In Corollary \ref{arcor}, if $A = A(G)\oplus \m C 1$ and $B = B_0(G)\oplus \m C 1$ such that the algebra $B_0(G)$ is regular, we have that $(B_0(G) \oplus \m C 1)^{-1}/\exp(B_0(G)\oplus \m C 1)$ is isomorphic to $(A(G) \oplus \m C 1)^{-1}/\exp(A(G) \oplus \m C 1)$ and hence $(B_0(G)\oplus \m C 1)^{-1} =\exp(B_0(G)+ \m C 1) (A(G) \oplus \m C 1)^{-1} $ for regular Rajchmann algebras $B_0(G)$. 
 \end{remark}
  The condition $\Delta(B_0(G)) = G$ is closely related with asymptotic properties of (strongly continuous) unitary representations $(\pi, \mathcal H_\pi)$ of a locally compact group $G$ and with the property of $square$-$integrability$. By the latter we mean that there is a dense subspace $D$ of $\mathcal{H}_\pi$, such that for all $\xi,\eta \in D$, the matrix coefficient $\phi_{\xi\eta}: G \to \m C, g \to \inn{\pi(g)\xi,\eta}$ is in $L^2(G)$. Square-integrable representations are $C_0$-representations, which means that all matrix coefficients of $\pi$ lie in $B_0(G)$, but not vice versa. It follows from the results of Rieffel \cite{rief}, Duflo-Moore \cite{duflo} and others that a representation $(\pi,\mathcal{H}_\pi)$ is square-integrable if and only if it is quasi-equivalent to a subrepresentation of the regular representation $(\lambda_G, L^2(G))$. We write $\pi\ \overset{q}{\leq} \ \lambda_G.$
  
  The following theorem proved in \cite{May3} shows that for certain real algebraic groups, every $C_0$-representation has a square integrable tensor power. In fact, recall that a connected real algebraic group $G$ has a unique largest unipotent radical $N$, and decomposes as $$ G = N \rtimes_\varphi H ,$$ where $H$ is a reductive Levi-complement of $N$ and $\varphi: H \to Aut(N)$ is a group homomorphism. The groups $N$ and $H$ are Zariski-closed, $N$ is simply connected with respect to the topology induced by $GL(n, \m C)$ and $H$ acts algebraically and reductively on the Lie algebra $n$ by the derived representation. The centralizer in $G$ of a subset $S \in G$ is denoted by $C_G(S)$. 
  
  \begin{theorem}(\cite[Theorem 1.1]{May3})\label{maythm}
  Let $G$ be a connected real algebraic group and keep the above notation. Suppose that \begin{enumerate}
      \item $C_G(H) \cap N =\{e\}$, i.e., $H$ acts non-trivially on $N$. 
      \item $C_G(\mathcal{Z}) \cap H$ is compact, i.e., $\{h \in H : \varphi(h)a=a,\, \forall\, a  \in \mathcal{Z}\}$ is compact, where $\mathcal{Z}$ is the center of $N$. 
  \end{enumerate}
  
  Then there exists a $k_0 \in \m N$ such that for all $k \geq k_0$ and for all representations $(\pi, \mathcal{H}_\pi)$ whose subrepresentations all have compact kernel 
  $$ \pi ^{\otimes k} \overset{q}{\leq} \lambda.$$
  \end{theorem}
  
  Several examples of groups are known with the property that for every $C_0$-representation $(\pi, \mathcal{H}_\pi)$, a sufficiently large tensor power is square-integrable: this follows for semisimple Kazhdan groups from the results of Cowling \cite{cowl1} and Moore \cite{moore}, for generalized motion groups from the results of Liukkonen and Mislove \cite{Liu} and for several connected totally minimal groups from the results of Mayer \cite{May1,May2,May3}.  
  
  Theorem \ref{maythm} bears significance to a conjecture of  Fig\`a-Talamanca and Picardello. The question of the square-integrability of tensor products is related to the relationship between $B_0(G)$ and the radicalizer $A_r(G)$ of the Fourier algebra. That is $$ A_r(G) = \{u \in B(G) : \exists k \in \m N \text{ such that }u^k \in A(G) \}.$$
  
 Fig\`a-Talamanca and Picardello \cite{Figpic} showed that $A_r(G)$ is not norm dense in $B_0(G)$ if the center of $G$ is not compact or if $G$ is non-compact nilpotent group. In particular, $A_r(G)$ is not norm dense in $B_0(G)$ if $G$ is a non-compact $abelian$ group. Their conjecture reads as follows:
 
 \begin{conjecture}
 Let $G$ be an analytic group with compact center and without non-compact nilpotent direct factors. Then $A_r(G)$ is dense in $B_0(G)$ i.e., $\Delta(B_0(G)) = G$. 
 \end{conjecture}
 
 The question of when $C_0$-representations are square-integrable has been widely studied. One of the most important results is that the left regular representation splits into irreducibles if every $C_0$-representation is square-integrable, i.e., $B_0(G) = A(G)$ (\cite{Fig}, \cite{Bag3}). The converse, though not true in general \cite{Bag2}, holds for many semidirect product groups, including the affine group and $p$-adic motion groups. This follows from (\cite{Nicrun}, Proposition 2.1). In \cite{knud1}, Knudby proves the following result: 
 \begin{theorem}(\cite[Theorem 4]{knud1})
 Let $G$ be a second countable locally compact group. Then $B_0(G) = A(G)$ provided that $G$ satisfies the following two conditions: $G$ is of type $I$, and there is a non-compact closed subgroup $H$ of $G$ such that every irreducible representation of $G$ is either trivial on $H$ or is a subrepresentation of the left regular representation.
 \end{theorem} 
 Knudby is able to prove that $B_0(P) = A(P)$ for the minimal parabolic subgroup of several simple Lie groups. In \cite{Knud2}, it is shown that there are uncountably many (non-isomorphic) second countable locally compact groups $G$ such that $B_0(G) = A(G)$ and $G$ has no nontrivial compact subgroups. Knudby also proves $B_0(G) = A(G)$ by weakening the above conditions: $G$ is type $I$ and there is a countable family $\mathrm{H}$ of non-compact closed subgroups $G$ such that each irreducible unitary representation of $G$ is either trivial on some $H \in \mathrm{H}$ or is a subrepresentation of the left regular representation of $G$. The advantage of this condition is that it is preserved under direct products. Khalil \cite{kha} showed that the $ax+b$ group, which is non-compact and solvable, satisfies $A(G) = B_0(G)$ and by the above mentioned result, if $G$ is direct product of the $ax+b$ group with itself, then $B_0(G) = A(G)$. 
 \begin{remark}\label{arcorr}
For a unital Banach algebra $A = I \oplus B$ where $I$ is an ideal and $B$ is an closed sub-algebra, we have $A^{-1} = (I\oplus \m C1)^{-1}B^{-1}$ because if $(u+v)(u'+v') = 1$, then $0=uu'+vu'+uv'$ $( \in I)$ and $vv' = 1$. So $u+v = (uv^{-1}+1)v$ with $(uv^{-1}+1)^{-1} = u'v+1 \in I\oplus \m C1$.
\end{remark}

The above remark allows us to immediately describe the structure of invertible elements in the following examples:

\subsection{Euclidean Motion groups} Let $G = \m R^d \rtimes SO(d), d \geq 2$, where $SO(d)$ acts on $\m R^d$ by rotation. 
\begin{theorem}(\cite[Theorem 4.1]{kanul})
\begin{enumerate}
    \item $B(G) = B_0(G) \oplus B(SO(d)) \circ q$, where $q : G \to SO(d)$ denotes the quotient homomorphism.
    \item $B_0(G) \neq A(G)$, but $B_0(G) = \{u \in B(G)\  | \ u^{4d-3} \in A(G)\}.$
    \item $\Delta(B_0(G)) = G$ (i.e. $B_0(G)$ is regular) and $\Delta(B(G)) = G \cup SO(d)$.
    \item $B_0(G)|_{\m R^d} \neq B_0(\m R^d)$.
\end{enumerate}
\end{theorem}

By Remark $\ref{arcorr}$ and $\ref{arcor}$, we have \begin{align*}
    B(G)^{-1} &= (B_0(G)\oplus \m C1)^{-1}B(SO(d))^{-1}\circ q\\
              &= \exp(B_0(G)\oplus \m C1) (A(G) \oplus \m C1)^{-1} (A(SO(d)))^{-1}\circ q
\end{align*} 
 In this case, $G_{z_F}$(Theorem \ref{Walty}, \ref{AP}), the almost periodic compactification of $G$, is $SO(d)$. 

\subsection{\texorpdfstring{$n^{\text{th}}$}{Lg} rigid \texorpdfstring{$p$}{Lg}-adic motion like groups} Let $G = A \rtimes K$ where \begin{enumerate}
    \item $K$ is a compact group acting on an abelian group $A$, with each of the groups separable, and
    \item the dual space $\widehat{G}$ is countable and decomposes as $\widehat{K}\circ q \sqcup \{\lambda_k\}_{k=1}^{\infty}$ where $\widehat{K}$ is the discrete dual space of $K$, $q:G \to K$ is the quotient map, and each $\lambda_k$ is a subrepresentation of the left regular representation. 
\end{enumerate} Then by (\cite[Proposition 2.1]{Nicrun}), we have $B(G) = A(K) \circ q \oplus^{l_1} A(G)$.  By Remark $\ref{arcorr}$, we have 
$$ B(G)^{-1} = (A(G) \oplus  \m C1)^{-1} (A(K)^{-1}\circ q).$$ 

Examples of such groups are the $n^{th}$ rigid $p$-adic motion group $G_{p,n}: = \m Q_p^n \rtimes GL(n,\m O_p)$, where $\m Q_p$ is the field of $p$-adic numbers, $\m O_p :=\{r \in \m Q_p : |r|_p \leq 1\}$ is the $p$-adic integers, which is a compact open subring of $\m Q_p$, and the compact group $GL(n, \m O_p)$, which is the multiplicative group of $n \times n$ matrices with entries in $\m O_p$ and determinant of valuation $1$ and, acts on the vector space $\m Q_p^n$ by matrix multiplication. For $n=1$, this is the $p$-adic motion group $\m Q_p \rtimes \m T_p$, where $\m T_p = \{r \in \m Q_p : |r|_p = 1\}$. The above decomposition for $B(G_{p,1})$ was proven independently in \cite{Wal3} and \cite{Mauc}. In this case, $G_{z_F}$ is $K$.

\subsection{The \texorpdfstring{$ax+b$}{Lg} group} Let $G = \{(a,b): a,b \in \m R, a > 0\}$ with multiplication $(a,b)(c,d) = (ac,ad+b)$. If $j: G \to \m R$ is the homomorphism given by $j(a,b) = \log a$, then $j$ is continuous with $\ker j = \{(1,b):b \in \m R\} \cong \m R$. In (\cite[Th\'eor\`eme 7]{kha}), it is shown that $$ B(G) = (B(\m R)\circ j) \oplus^{l_1}A(G).$$ By Remark $\ref{arcorr}$, $$ B(G)^{-1} = (A(G)\oplus \m C1)^{-1}(B(\m R)^{-1}\circ j).$$ The groups continuously isomorphic to $\m R$ are just $\m R$ and $\m R_d$ (reals with the discrete topology). Hence by (\cite[Theorem 3]{Tay1}), each $\mu \in M(\m R)^{-1}$ has the form $\mu = \nu_1 * \nu_2 * e^w$ for $\nu_1 \in (L^1(\m R) \oplus \m C1)^{-1}, \nu_2 \in L^1(\m R_d)$ and $w \in M(\m R)$. Via the Fourier-Stieltjes transform, this factorization of invertible elements in $B(\m R)$ has the following form: each $u \in B(\m R)^{-1}$ has the form $u = v_1v_2e^w$ for $v_1 \in (A(\m R) \oplus \m C1)^{-1},v_2 \in A(\m R^{ap})^{-1}$ and $w \in B(\m R)$, Here $\m R^{ap}$ is the almost periodic compactification of $\m R$, which is also the Pontraygin dual group of $\m R_d$. Hence for every $u \in B(G)^{-1}$, we have $u = (v_1\circ j)\cdot (v_2\circ j)\cdot v_3e^w$ for $v_1 \in (A(\m R) \oplus  \m C1)^{-1},v_2 \in A(\m R^{ap})^{-1}, v_3 \in (A(G) \oplus  \m C1)^{-1}$ and $w \in B(G)$. In this case, $G_{z_F}$ is $\m R^{ap}$.

\section{Main Theorem}\label{main}

The study of invertibles in Fourier-Stieltjes algebras of classical motion groups instigated the search for a reasonable structure theory of Fourier-Stieljes algebra with its spectrum being a semilattice of groups. Motivated by examples of such Fourier-Stieltjes algebra (Section \ref{decom}), we present our main theorem in this section. We note that the $z$'s appearing in the following theorem are, indeed, critical elements of $\Delta(B(G))_+$ in the sense of (\cite[pp. 275]{Wal2}) and hence the corresponding maximal group $G_z$ in $\Delta(B(G))$ at the critical element $z$ is a locally compact group (\cite[pp. 276]{Wal2}). 

In what follows, $z' \leq z$ means $z'z = z'$.  
The following theorem is an adaptation of (\cite[Proposition 4.5]{Tay1}).

\begin{theorem}\label{mainthm}
Let $G$ be locally compact group such that the Fourier-Stieltjes algebra $B(G)$ with the maximal ideal space $\Delta(B(G))$ admits the following decomposition:
\begin{enumerate}
    \item $B(G) = \bigoplus^{l_1} \left\{\overline{A_r(G_z)}\ \middle|\ z \in \Delta(B(G))_+\right\}$,
    \item each $\overline{A_r(G_z)}$ is a closed ideal in $\bigoplus^{l_1} \left\{\overline{A_r(G_{z'})} \ \middle| \ z' \leq z \right\}$.
    \item $\Delta(B(G)) =  \bigcup\left\{G_z \ \middle| \  z \in \Delta(B(G))_+\right\}$ where $G_z = \{s \in \Delta(B(G)): s^*s = ss^* = z\}$. 
\end{enumerate} 
 Then for any $u \in B(G)^{-1}$, we have that $$ u = v_1\cdot v_2\cdots v_n\cdot\exp(w)$$ with each $v_i \in (A(G_{z_i})\oplus \m C 1)^{-1}$ for some $z_i \in \Delta(B(G))_+$ and $w \in B(G)$. The elements $v_i$ are unique modulo $\exp(A(G_{z_i})+\m C 1)$.
\end{theorem}
\begin{proof}
The union of the finite sums of these subalgebras, $\overline{A_r(G_z)}$ for $z \in \Delta_+$, is dense in $B(G)$. So given $u \in B(G)^{-1}$, find $\psi$ in one of those finite sums such that $\|u - \psi\| \leq \|u^{-1}\|^{-1}$. Then $\psi \in B(G)^{-1}$ and $u = \psi * \exp(\rho)$ for some $\rho \in B(G)$. Hence, we may assume $u \in A(G_{z_F}) + \overline{A_r(G_{z_1})} + \overline{A_r(G_{z_2})} + \cdots + \overline{A_r(G_{z_n})}$ for some $z_1,z_2,\ldots,z_n \in \Delta_+$.  We can also assume that the almost periodic component of $u$ is $1$, because if $u \in B(G)^{-1}$, then $z_F\cdot u \in A(G_{z_F})^{-1}$ and we can replace $u$ by $(z_F\cdot u)^{-1}u$.

Let $u = 1+ \psi_{z_1} +\psi_{z_2} +\cdots+ \psi_{z_n}$ , where each $\psi_{z_i} \in \overline{A_r(G_{z_i})}$ and $z_1$ can be chosen minimal among $z_i$. Now $z_1\cdot u = 1+\psi_{z_1}$ is invertible in $\overline{A_r(G_{z_1})}\oplus\mathbb{C}1$, and $(1+\psi_{z_1})^{-1}(1+ \psi_{z_1} +\psi_{z_2} +\cdots +\psi_{z_n}) = 1+\psi_{z_2}'+\psi_{z_3}' + \cdots + \psi_{z_n}'$ where $\psi_{z_j}' = (1+\psi_{z_1})^{-1}\psi_{z_j}$ for $j=2,\ldots,n$ and $\psi_{z_j}'\perp \overline{A_r(G_{z_1})}$. By means of induction, we obtain the following factorization: if $u \in B(G)^{-1},$ then $u = \psi_{1}\cdot \psi_{2}\cdots \psi_{n}\cdot\exp(\rho)$ with $\rho \in B(G)$ and each $\psi_{i} \in (\overline{A_r(G_{z_i})} \oplus \m C 1)^{-1}$ for some $z_i \in \Delta_+$.

By Remark \ref{reduc}, each $\psi_i$ is of the form $ v_i \cdot \exp(\rho_i)$, where $v_i \in (A(G_{z_i})\oplus\m C1)^{-1}$ and $\rho_i \in (\overline{A_r(G_{z_i})} \oplus \m C 1)$. Now let $w = \rho_1+\cdots+\rho_n+\rho$.
\end{proof}

\section{Examples} \label{decom}
In this section, we describe some classes of locally compact groups $G$ for which the spectrum of their Fourier-Stieltjes algebras is a semilattice of groups and satisfies the hypotheses of our main theorem $\ref{mainthm}$.
 
 \subsection{Compact extensions of Nilpotent groups} Suppose $G$ is a connected lie group which has a closed normal nilpotent subgroup $M$ such that $G/M$ is compact and let $\m T^p$ be the maximal torus in the center of the connected component $N = M_0$ of $M$. For each subgroup $H$ with $\m T^p \subset H \subset G$, let $\overline{H} = H/\m T^p$. Consider the following conditions:
 \begin{enumerate}
     \item $\overline{G} = \overline{N} \rtimes K$, where $K$ is a compact analytic group, and
     \item the action of $K$ on the Lie algebra $L(\overline{N})$ has no nonzero fixed points.
 \end{enumerate}

Let $X$ be the set of (necessarily closed) $G$-normal analytic subgroups of $N$ which contain $\m T^p$, together with the trivial subgroup $\{e\}$. We consider $\m T^p \notin X$.
 Liukkonen and Mislove (\cite[Section 3]{Liu} ) proved that for such groups $B(G)$ is symmetric if and only if the above conditions hold and their analysis is detailed enough to yield the following decomposition of $B(G)$ and description of its maximal ideal space $\Delta$ :
 
\begin{enumerate}
    \item $B(G) = \bigoplus^{l_1} \left\{\overline{A_r(G/V)} \middle| V \in X\right\}$
    \item $\overline{A_r(G/V_1)}\cdot\overline{A_r(G/V_2)} \subset \overline{A_r(G/V_1 \cap V_2)}$.
    \item each $\overline{A_r(G/V)}$ is a closed ideal in $\bigoplus^{l_1} \left\{\overline{A_r(G/W)} \middle| W \supseteqq V\right\}$.
\end{enumerate}

 \begin{theorem}(\cite[Theorem 3.1]{Liu})\label{liu}
 Suppose $G$ is an analytic group such that $G/\mathbb{T}^p = N/\mathbb{T}^p \rtimes K$, where $K$ is compact analytic, $N$ is nilpotent analytic, and $\m T ^p$ is the maximal torus in $Z(N)$. Suppose the action of $K$ on $L(N/\m T^p)$ has no nonzero fixed points. Then the maximal ideal space $\Delta$ of $B(G)$ satisfies $$ \Delta = \{G/V : V \in X\}.$$  Multiplication in $\Delta$ is given by $g_1V_1g_2V_2 = g_1g_2 V_1V_2$. Convergence may be described as follows: if $g_nV_n \to gV$ in $\Delta$, then eventually $V_n \subseteqq V$ and $g_nV \to gV$ in $G/V$. Conversely, given a sequence $\{g_nV_n\}\subset \Delta$, there is a unique smallest dimensional $V$ such that $V_n \subset V$ eventually and $\{g_nV_n\}$ is eventually bounded in $G/V$; if $g_nV \to gV$ in $G/V$, then $g_nV_n \to gV$ in $\Delta$.
 \end{theorem}
 
 \subsection{Connected semisimple Lie groups with finite center} Let $G$ be a connected semisimple Lie group with finite centre. Then we may decompose $G = (G_0 \times G_1 \times \cdots G_n)/C $ where $G_0$ is compact analytic, all of $G_1,\ldots,G_n$ are non-compact and simple analytic and $C$ is a finite central group of the product. Now let $\mathcal{S}$ be the set of groups $G_F = \prod_{k \in F} G_k/C_F$ where $F$ is a (possibly empty) subset of $\{1,\ldots,n\}$ and $C_F = C \cap \prod_{k \in F} G_k$. 
 
 \begin{theorem}(\cite[pp. 90]{Cowl})\label{Cow}
 Let $G$ be a connected semisimple Lie group with finite centre and let $\mathcal{S}$ be as described above. Then the maximal ideal space $\Delta$ of $B(G)$ is $$ \Delta = \bigcup_{S \in \mathcal{S}} G/S,$$ $G$ is dense in $\Delta$ and $B(G)$ is symmetric. 
 \end{theorem}
 
\subsection{Totally minimal groups} We call a group $G$ $totally$ $minimal$ if for any closed normal subgroup $N$, the factor group $G/N$ admits no Hausdorff group topology which is strictly coarser than the quotient topology. 

\begin{theorem}(\cite[Theorem 2.5]{May2})\label{may}
For a connected locally compact group $G$, the following are equivalent:
\begin{enumerate}
    \item $G$ is totally minimal;
    \item there exists a compact normal subgroup $K \triangleleft G$ such that the factor group satisfies
    $$ G/K = N \rtimes H,$$ where $N$ is a simply connected nilpotent Lie group, $H$ is connected linear reductive and $H$ operates on $N$ without nontrivial fixed points.
\end{enumerate}
\end{theorem}

Consider a topological commensurability relation for a closed normal subgroup : $S \sim S'$ if and only if $S/(S \cap S')$ and $S'/(S \cap S')$ are compact. We may consider certain minimal representative of each commensurability class:
$$ \underline{S} = \bigcap\{S' : S' \text{ is closed normal subgroup of } G \text{ with } S \sim S' \}.$$ Let $\mathcal{N}(G)$ denote the collection of these minimal non-commensurable representatives. 

Given $G,K,N$ and $H$ be as in the theorem above, let $$ \mathcal{N}_H(N) = \{S \subseteq N : S \text{ is a connected normal $H$-invariant subgroup}\}.$$ Now decompose $H = (H_0 \times H_1 \times \cdots \times H_n )/C$ where $H_0$ is compact linear group and all $H_1,\ldots,H_n$ are non-compact simple linear Lie groups, and $C$ is a finite central subgroup of the product. We have that $\mathcal{N}(H)$ consists of groups $H_F = \prod_{k \in F} H_k/C_F$ where $F$ is a (possibly empty) subset of $\{1,\ldots,n\}$ and $C_F = C \cap \prod_{k \in F} H_k$. Now as in Remark 2.2 of \cite{Nico}, we get the following structure of $B(G)$ which is a reformulation of  Mayer (\cite[Theorem 15]{May1}): 

$$ B(G) = \bigoplus^{l_1}_{\substack{S \in \mathcal{N}_H(N)\\ F \subseteq\{1,\ldots,n\}}} B_0(G/(S \rtimes H_F))$$

\begin{remark}
By Theorem $\ref{liu}$, if $G$ a is linear connected totally minimal with $H$ compact, then the maximal ideal space of $B(G)$ is a semilattice of groups.
\end{remark}

\begin{remark}
By Theorem $\ref{Cow}$ and (\cite[Corollary 2.12]{May3}), if $H$ is a connected linear reductive group, then the maximal ideal space $\Delta$ of $B(H)$ is 
$$ \Delta = \bigcup_{H_F \in \mathcal{N}(H)} H/H_F$$ and $B(H)$ is symmetric. 
\end{remark}

\begin{remark} If $G$ is a linear connected totally minimal group with $N$ abelian, so are its quotients. By Theorem $\ref{may}$, $G = V \rtimes H$ for a vector group $V$. By (\cite[Example 2.9]{May3}), $B_0(G) = \overline{A_r(G)}$ and hence $B_0(G/S) = \overline{A_r(G/S)}$ for any closed normal subgroup $S$ of $G$. Therefore in this case, we get that $B(G)$ is symmetric and its maximal ideal space
$$ \Delta = \{G/(S \rtimes H_F): S \in \mathcal{N}_H(V), F \subseteq\{1,\ldots,n\}\}.$$
\end{remark}

\subsection*{Acknowledgements} The author would like to thank his supervisor Nico Spronk for innumerable valuable discussions and for suggesting this problem. The author also wishes to thank both his advisors, Nico Spronk and Brian Forrest for their constant encouragement and advice. We thank the referees for their valuable organizational suggestions. 

\bibliography{main}
\bibliographystyle{amsalpha}

\baselineskip=17pt


\end{document}